\documentclass{amsart}

\usepackage[english]{babel}
\usepackage{amssymb,amscd,amsthm,latexsym}
\usepackage{hyperref}
\usepackage[all]{xy}
\usepackage{color}

\theoremstyle{plain}
\newtheorem{theorem}[subsection]{Theorem}
\newtheorem{lemma}[subsection]{Lemma}
\newtheorem{prop}[subsection]{Proposition}

\theoremstyle{definition}
\newtheorem{defi}[subsection]{Definition}
\newtheorem{exe}[subsection]{Example}
\newtheorem{exs}[subsection]{Examples}
\newtheorem{rem}[subsection]{Remark}
\newtheorem{exercices}[subsection]{Exercices}
\newtheorem{ex}[subsection]{Exercice}

\newcounter{numexo}
\newcommand{\ba}{\begin{array}}
\newcommand{\ben}{\begin{enumerate}}
\newcommand{\bt}{\begin{tabular}}

\newcommand{\ea}{\end{array}}
\newcommand{\ed}{\end{diagram}}

\newcommand{\een}{\end{enumerate}}
\newcommand{\et}{\end{tabular}}

\newcommand{\bit}{\begin{itemize}}
\newcommand{\eit}{\end{itemize}}

\newcommand{\Eq}{\ensuremath{\mathsf{Eq}}}

\newcommand{\Ens}{\ensuremath{\mathsf{Set}}}

\newcommand{\Set}{\ensuremath{\mathsf{Set}}}

\newcommand{\Top}{\ensuremath{\mathsf{Top}}}

\newcommand{\Mod}{\ensuremath{\mathsf{Mod}}}
\newcommand{\Ann}{\ensuremath{\mathsf{Rng}}}

\renewcommand{\ker}{\ensuremath{\mathsf{ker\,}}}

\newcommand{\Cc}{\ensuremath{\mathbb{C}}}

\newcommand{\Ab}{\ensuremath{\mathsf{Ab}}}
\newcommand{\Grp}{\ensuremath{\mathsf{Grp}}}

\newcommand{\Gpt}{\ensuremath{\mathsf{Grp(Top)}}}

\newdir{|>-}{!/  4.5pt / @{ }*@{ }*@{ }*!/ 0pt / @{|}*!/ -4.5pt / :(1,-.2)@^{>}*!/ -4.5pt / :(1,+.2)@_{>}}
\newdir{-|>}{!/ 9pt / @{|}*!/ 4.5pt /  :(1,-.2)@^{>}*!/ 4.5pt /  :(1,+.2)@_{>}*@{ }*@{ }*@{ } }
\newdir{ >}{!/  8.5pt / @{ }*@{ }*@{ }*@{>}}
\newdir{>}{{}*:(1,-.2)@^{>}*:(1,+.2)@_{>}}
\newdir{<}{{}*:(1,+.2)@^{<}*:(1,-.2)@_{<}}

\newdir{ >}{
  @{}*!/-10pt/@{>} }

\newdir{|>-}{!/  4.5pt / @{ }*@{ }*@{ }*!/ 0pt / @{|}*!/ -4.5pt / :(1,-.2)@^{>}*!/ -4.5pt / :(1,+.2)@_{>}}
\newdir{-|>}{!/ 9pt / @{|}*!/ 4.5pt /  :(1,-.2)@^{>}*!/ 4.5pt /  :(1,+.2)@_{>}*@{ }*@{ }*@{ } }
\newdir{ >}{!/  8.5pt / @{ }*@{ }*@{ }*@{>}}
\newdir{>}{{}*:(1,-.2)@^{>}*:(1,+.2)@_{>}}
\newdir{<}{{}*:(1,+.2)@^{<}*:(1,-.2)@_{<}}
\newdir{ >}{{}*!/-10pt/\dir{>}}

\newbox\pullbackbox
\setbox\pullbackbox=\hbox{\xy 0;<1mm,0mm>: \POS(4,0)\ar@{-} (0,0) \ar@{-} (4,4)
\endxy}

\begin{document}
\title{\emph{An introduction to regular categories} }
\author{Marino Gran}
\date{8 April, 2020}
\maketitle
\begin{abstract}
{This paper provides a short introduction to the notion of regular category and its use in categorical algebra. We first prove some of its basic properties, and consider some fundamental algebraic examples. We then ana-lyse the algebraic properties of the categories satisfying the additional Mal'tsev axiom, and then the weaker Goursat axiom. These latter contexts can be seen as the categorical counterparts of the properties of $2$-permutability and of $3$-permutability of congruences in universal algebra. Mal'tsev and Goursat categories have been intensively studied in the last years: we present here some of their basic properties, which are useful to read more advanced texts in categorical algebra.}
\end{abstract}
\section*{Introduction}
In categorical algebra some structural properties of varieties of universal algebras are investigated by replacing the arguments involving elements of an algebraic structure and its operations with other ones using relations and commutative diagrams. A typical example is provided by the study of \emph{Mal'tsev categories} \cite{CLP}, which can be seen as the categorical counterpart of \emph{Mal'tsev varieties} (in the sense of \cite{Smith}), also called $2$-permutable varieties in the literature. Instead of requiring the existence, in the algebraic theory of the variety, of a ternary term $p(x,y,z)$ verifying the identities $p(x,y,y)=x$ and $p(x,x,y)= y$, one asks that any internal reflexive relation in the category is an equivalence relation. This categorical property, with its many equivalent formulations, has turned out to be strong enough to establish, in the regular context, many of the well known properties of Mal'tsev varieties (see \cite{BGJ} for a recent survey on the subject, and the references therein).

This survey article can be seen as a first introduction to the basic categorical notions which are useful to express the exactness properties of various kinds of algebraic varieties in the sense of universal algebra. The main goal of this text is to introduce the reader to the notion of \emph{regular category}, which is fundamental in category theory, since abelian categories, elementary toposes and varieties of universal algebras are all regular categories. 
Special attention will be paid to the so-called \emph{calculus of relations}, which provides a powerful method to prove new results in regular categories, possibly satisfying some additional exactness conditions. A good knowledge of the fundamentals of regular categories are useful to understand many of the recent developments in categorical algebra.
 The Mal'tsev axiom gives the opportunity to illustrate this method: in a regular category this axiom is equivalent to the permutability of the composition of relations, in the sense that any pair $R$ and $S$ of equivalence relations on a given object are such that $R \circ S = S \circ R$.
  Some recent results concerning the more general \emph{Goursat categories} \cite{CKP, GRT} will then be explained in the last section.
  These aspects are useful to illustrate many of the links between exactness properties in categorical algebra, the so-called Mal'tsev conditions in universal algebra, and the validity of suitable homological lemmas \cite{Lack, GR, GR2}.
  \vspace{3mm}
  
{\bf Acknowledgement.} A part of the material presented in this survey article is based on \cite{BG, Gran-EPFL, BGJ}. The author is grateful to Tomas Everaert for an important suggestion concerning Theorem \ref{BK}. Many thanks also to Diana Rodelo, Idriss Tchoffo Nguefeu and David Broodryk for carefully proofreading a first version of the article.
\section{Regular categories}
The notion of regular category plays an important role in the categorical understanding of algebraic structures. Regular categories capture some fundamental exactness properties shared by the categories $\Ens$ of sets, $\Grp$ of groups, $\Ab$ of abelian groups, R-$\Mod$ of modules on a commutative ring R and, more generally, by any variety $\mathbb V$ of universal algebras. Topological models of ``good'' algebraic theories, such as the categories $\mathsf{Grp(Top)}$ of topological groups and $\mathsf{Grp(Comp)}$ of compact Hausdorff groups are also regular. Other examples will be considered later on in Examples \ref{regular-ex} and \ref{exemples-M}. The basic idea is that any arrow in a regular category can be factorized through an (essentially unique) \emph{image}, and that these factorizations are stable under pullbacks.

Regular categories also have a prominent role in categorical logic (see \cite{Johnstone-elephant}, for instance, and the references therein). However, in this introductory course we shall only focus on the algebraic examples, with the goal of illustrating the importance of regular categories in categorical algebra.

In order to understand the notion of regular category it is useful to compare a few types of epimorphisms: this will be the subject of the following section (see \cite{BG} for further details).
\subsection*{Strong and regular epimorphisms}




\begin{defi}
An arrow $f \colon A \rightarrow B$ in a category $\mathbb C$ is a \emph{strong epimorphism} if, given any commutative square
$$
\xymatrix{ A \ar[r]^{f} \ar[d]_g & B \ar[d]^h \ar@{.>}[dl]_t \\
C \ar[r]_m & D
}
$$
in $\mathbb C$, where $m \colon C \rightarrow D$ is a monomorphism, there exists a unique arrow $t \colon B \rightarrow C$ such that $m \circ t = h$ and $t \circ f= g$.
\end{defi}

\begin{rem}\label{Epi fort-epi}
If the category $\mathbb C$ has binary products, then every strong epimorphism is an epimorphism. Indeed, if $f \colon A \rightarrow B$ is a strong epimorphism, and $u,v \colon B \rightarrow C$ are two arrows such that $u \circ f =v \circ f$, one can then consider the diagonal $(1_C, 1_C)= \Delta \colon C \rightarrow C \times C$ and the commutative square
$$
\xymatrix{ A \ar[r]^{f} \ar[d]_{u \circ f =v \circ f} & B \ar[d]^{(u,v)} \\
C \ar[r]_-{\Delta} & C \times C.
}
$$
Since $\Delta$ is a monomorphism, there is a unique $t \colon B \rightarrow C$ such that $\Delta \circ t = (u,v)$ and $t \circ f = u \circ f =v \circ f$. It follows that $$u= p_1 \circ (u,v)  = p_1 \circ \Delta \circ t = t = p_2 \circ \Delta \circ t = p_2 \circ (u,v)= v,$$
where $p_1 \colon C \times C \rightarrow C$ and $p_2 \colon C \times C \rightarrow C$ are the product projections.
\end{rem}

\begin{lemma}
An arrow $f \colon A \rightarrow B$ is an isomorphism if and only if $f\colon A \rightarrow B$ is a monomorphism and a strong epimorphism.
\end{lemma}
\begin{proof}
If $f$ is both a strong epimorphism and a monomorphism, one considers the commutative square
$$
\xymatrix{ A \ar[r]^{f} \ar[d]_{1_A} & B \ar[d]^{1_B} \ar@{.>}[dl]_t \\
A \ar[r]_-{f} & B.
}
$$
The unique (dotted) arrow $t \colon B \rightarrow A$ such that $f \circ t = 1_B$ and $t \circ f= 1_A$ is the inverse of $f$. The converse implication is obvious. \end{proof}

\begin{exercices}\label{exo-strong}
Prove that strong epimorphisms are closed under composition, and that, if the composite $g \circ f$ of two composable arrows is a strong epimorphism, then $g$ is a strong epimorphism.
Show that if $g \circ f$ is a strong epimorphism, with $g$ a monomorphism, then $g$ is an isomorphism.
\end{exercices}

\begin{defi}
An arrow  $f \colon A \rightarrow B$ is a \emph{regular epimorphism} if it is the coequalizer of two arrows in $\mathbb C$. 
\end{defi}

\begin{defi}
A  \emph{split epimorphism} is an arrow $f \colon A \rightarrow B$ such that there is an arrow $i \colon B \rightarrow A$ with $f \circ i = 1_B$.
\end{defi}
Observe that the axiom of choice in the category $\Ens$ says precisely that any epimorphism is a split epimorphism. This is not the case in the categories $\Grp$ of groups or $\Ab$ of abelian groups, for instance. We are now going to prove the following chain of implications:
\begin{prop}\label{epis}
Let $\mathbb C$ be a category with binary products. One then has the implications

\centerline{\rm split epimorphism $\Rightarrow$ regular epimorphism $\Rightarrow$ strong epimorphism $\Rightarrow$ epimorphism}
\end{prop}

\begin{proof}
If $f \colon A \rightarrow B$ is split by an arrow $i \colon B \rightarrow A$, then $f$ is the coequalizer of  $1_A$ and $i \circ f$. Indeed, one sees that
$f \circ (i \circ f )= f = f \circ 1_A.$
Moreover, if $g \colon A \rightarrow X$ is such that $g \circ( i \circ f) = g \circ 1_A$, then $\phi = g \circ i$ is the only arrow with the property that $\phi  \circ f = g$.
 
 \vspace{3mm}
 
 Assume then that $f \colon A \rightarrow B$ is a regular epimorphism. It is then the coequalizer of two arrows, say $u \colon T \rightarrow A$ and $v \colon T \rightarrow A$:
 consider the commutative diagram
 $$
\xymatrix{  A \ar[r]^{f} \ar[d]_g & B \ar[d]^h \\
C \ar[r]_m & D
}
$$
 with $m$ a monomorphism. The equalities $$m \circ g \circ u = h \circ f \circ u = h \circ f \circ v = m \circ g \circ v$$
imply that $g \circ u = g \circ v$. The universal property of the coequalizer $f$ implies that there is a unique $t \colon B \rightarrow C$ such that $t \circ f = g$. This arrow $t$ is also such that $m \circ t = h$, so that $f$ is a strong epimorphism.

\vspace{3mm} The fact that any strong epimorphism is an epimorphism when $\mathbb C$ has binary products has been shown in Remark \ref{Epi fort-epi}.
\end{proof}

\subsection*{Quotients in algebraic categories}
Let us then consider some examples of \emph{quotients} in the categories of sets, of groups and of topological groups, which will be useful to explain the general construction in regular categories.

Let $f \colon A \rightarrow B$ be a map in $\Ens$, and $$\Eq(f)= \{(x,y) \in A \times A \mid f(x)=f(y) \}$$ its \emph{kernel pair}, i.e. the equivalence relation on $A$ identifying the elements of $A$ having the same image by $f$.
 This equivalence relation can be obtained by building  the pullback of $f$ along $f$:
\begin{equation}\label{kernel-pair}
\xymatrix{\Eq(f) \ar[r]^{p_2} \ar[d]_{p_1} & A \ar[d]^f \\
A \ar[r]_f & B.
}
\end{equation}
\begin{ex}\label{pairenoyau}
Show that  any regular epimorphism $f$ in a category with kernel pairs is the coequalizer of its kernel pair $(\Eq(f), p_1, p_2)$.
\end{ex}
In the category $\Ens$ of sets one sees that
 the canonical quotient $\pi \colon A \rightarrow A/\Eq(f)$ defined by $\pi(a) = \overline{a}$ is the coequalizer of $p_1$ and $p_2$. This yields a unique arrow $i \colon 
 A/\Eq(f) \rightarrow B$ such that $i \circ \pi = f$:
$$ \xymatrix{\Eq(f) \ar@<2pt>[r]^-{p_1}\ar@<-2pt>[r]_-{p_2} &A \ar[dr]_{\pi}  \ar[rr]^f & & B \\
& & A/\Eq(f) \ar@{.>}[ru]_i & }$$
The map $i$ is defined by $i(\overline{a})=f(a)$ for any $\overline{a} \in A/\Eq(f)$. This gives a factorization $i \circ \pi$ of the arrow $f$, where $\pi$ is a regular epimorphism (= a surjective map) and $i$ is a monomorphism (= an injective map) in the category $\Ens$.

The same construction is possible in the category $\Grp$ of groups. Indeed, given a group homomorphism $f \colon G \rightarrow G'$, one can consider the kernel pair $\Eq(f)$  which is again obtained by the pullback \eqref{kernel-pair} above, but this time computed in the category $\Grp$. The equivalence relation $\Eq(f)$ is a group, as a subgroup of the product $G \times G$ of the group $G$ with itself. 
The canonical quotient $\pi \colon G \rightarrow G / \Eq(f)$ is a group homomorphism, and this allows one to build the following commutative diagram in $\Grp$
$$ \xymatrix{\Eq(f) \ar@<2pt>[r]^-{p_1}\ar@<-2pt>[r]_-{p_2} &G \ar[dr]_{\pi}  \ar[rr]^f & & G' \\
& & {G/\Eq(f),} \ar@{.>}[ru]_i & }$$
where $\pi$ is a regular epimorphism and $i$ is a monomorphism, exactly as in $\Set$.

In the category $\Gpt$ of topological groups, where the arrows are continuous homomorphisms, it is again possible to obtain the same kind of factorization regular epimorphism-monomorphism for any arrow. We write $(G, \cdot, \tau_G)$ for a topological group, where $\tau_G$ is the topology making both the multiplication $\cdot$ and the inversion of the group continuous.
Given a continuous homomorphism $f \colon (G, \cdot, \tau_G) \rightarrow (G', \cdot, \tau_{G'})$ in $\Gpt$, the kernel pair $(\Eq(f), \cdot, \tau_i)$ is a topological group for the topology $\tau_i$ induced by the product topology $\tau_{G \times G}$ of the topological group $(G \times G, \cdot, \tau_{G \times G})$.
At the algebraic level the quotients in $\Gpt$ are actually computed as in $\Grp$, and then equipped with the quotient topology $\tau_q$. In this way one gets the following commutative diagram 
$$ \xymatrix{(Eq(f), \cdot, \tau_i) \ar@<2pt>[r]^-{p_1}\ar@<-2pt>[r]_-{p_2} &(G, \cdot, \tau_G) \ar[dr]_{\pi}  \ar[rr]^f & & (G',\cdot, \tau_{G'}) \\
& & {(G/\Eq(f), \cdot, \tau_q) } \ar@{.>}[ru]_i & }$$
where $\pi$ is the canonical quotient.
It turns out that $\pi$ is the coequalizer of the projections $p_1$ and $p_2$ in $\Gpt$, and the induced arrow
 $$i \colon (G/\Eq(f), \cdot, \tau_q) \rightarrow (G',\cdot, \tau_{G'})$$ is a monomorphism (since it is injective). Note that this factorization is not the one where the direct image $f(G)$ of the continuous homomorphism is equipped with the subspace topology induced by the topology of $(G',\cdot, \tau_{G'})$.
\vspace{3mm}

There are many other categories where the same construction as in $\Ens, \Grp$ and $\Gpt$ is possible, for instance in the category $\Ann$ of rings, $\mathsf{Mon}$ of monoids, $\Ab$ of abelian groups and, more generally, in any variety $\mathbb V$ of universal algebras.

All these are examples of regular categories in the following sense:
\begin{defi}\label{def1}\cite{Barr}
A finitely complete category $\Cc$ is \emph{regular} if
\begin{itemize}
\item coequalizers of kernel pairs exist in $\Cc$;
\item regular epimorphisms are pullback stable in $\Cc$.
\end{itemize}
\end{defi}
\begin{exs}\label{regular-ex}
\begin{itemize}
\item The category ${\Ens}$ is regular. We have observed that the coequalizers of kernel pairs exist in $\Ens$, and it remains to check the pullback stability of regular epimorphisms. 
Consider a pullback
$$\xymatrix{E \times_B A  \ar[r]^-{\pi_2} \ar[d]_{\pi_1}& A \ar[d]^{f} \\
E \ar[r]_p & B
}$$
in $\Ens$ where $p$ is a surjective map  (i.e. a regular epimorphism), and let us show that $\pi_2$ is also surjective. Let $a$ be an element in $A$; there exists then an  $e \in E$ such that $p(e)=f(a)$. This shows that there is an $(e,a) \in E \times_B A$ such that $\pi_2 (e, a) = a$, and $\pi_2$ is surjective. The same argument still works in the category $\Grp$ of groups, by taking into account the fact that regular epimorphisms therein are precisely the surjective homomorphisms, and that pullbacks are computed in $\Grp$ as in $\Ens$. For essentially the same reason the categories $\Ann$ of rings, $\mathsf{Mon}$ of monoids, and R-$\Mod$ of modules on a ring $R$ are also regular categories. More generally, any variety $\mathbb V$ of universal algebras is a regular category, any quasivariety - such as the category $\mathsf{{Ab}_{t.f.}}$ of torsion-free abelian groups - and also any category monadic over the category of sets, as for instance the category $\mathsf{CHaus}$ of compact Hausdorff spaces, and the category $\mathsf{Frm}$ of frames.
\item The category $\Gpt$ of topological groups is regular \cite{CKP}. The main point here is that the canonical quotient $\pi \colon (H, \cdot, \tau_H) \rightarrow (H/\Eq(f), \cdot, \tau_q)$ of a topological group $(H, \cdot, \tau_H)$ by the equivalence relation $(\Eq(f), \cdot, \tau_i)$ which is the kernel pair of an arrow $f \colon (H, \cdot, \tau_H) \rightarrow (G, \cdot, \tau_G)$ in $\Gpt$ is an \emph{open surjective homomorphism}. 
To check this latter fact, let us write $K = \ker(\pi)$ for the kernel of $\pi$, and let us then show that $$\pi^{-1} (\pi (V)) = V \cdot K$$ for any open $V \in \tau_H$. On the one hand if  $z= v\cdot k$, where $v \in V$ and $k \in K$, one has $$\pi(z) = \pi (v\cdot k) = \pi(v)\cdot  \pi (k) = \pi(v) \in \pi (V),$$ so that $z \in \pi^{-1} (\pi(V))$.
Conversely, if $z \in  \pi^{-1} (\pi(V))$, then $\pi(z) = \pi (v_1)$, for a $v_1 \in V$, so that $v_1^{-1} \cdot z \in K$, and $z = v_1 \cdot k$, for a $k \in K$. \\
This implies that
$$\pi^{-1}( \pi(V) )= (\bigcup_{k \in K} V\cdot k ) \in \tau_H.$$
Indeed, the function $m_k \colon G \rightarrow G$ defined by $m_k (x)= x\cdot k$ for any $x \in G$ (with fixed $k \in K$) is a homeomorphism, hence $V \cdot k = m_k (V)  \in \tau_H$, since $V\in \tau_H.$
We have then shown that $\pi(V)$ is open for any $V \in \tau_H$, and the map $\pi$ is open. It follows that in $\Gpt$ the regular epimorphisms are the open surjective homomorphisms.
To conclude that $\Gpt$ is a re-gular category it suffices to recall that the open surjective homomorphisms are pullback stable (a well known fact which can be easily checked). More generally, the models of any Mal'tsev theory in the category of topological spaces, i.e. any category of topological Mal'tsev algebras, is a regular category \cite{J-P}. Notice that also the category $\mathsf{Grp(Haus)}$ of Hausdorff groups, or $\mathsf{Grp(Comp)}$ of compact Hausdorff groups are also regular \cite{BC}.
\item As mentioned in the introduction any abelian category \cite{Buch} is a regular category, as is any elementary topos \cite{Johnstone-elephant}.
\item The category $\Top$ of topological spaces, unlike $\Gpt$, is not regular. The main reason is that in $\Top$ regular epimorphisms are quotient maps, and these are not pullback stable (see \cite{Borceux} for a counter-example, for instance).
\end{itemize}
\end{exs}
We are now going to show that any arrow in a regular category has a canonical factorization as a regular epimorphism followed by a monomorphism, exactly as in the examples  of the categories $\Ens, \Grp$ and $\Gpt$ recalled here above. 

\begin{theorem}\label{theo1} Let $\Cc$ be a regular category. Then
\begin{enumerate}
\item any arrow $f \colon A \rightarrow B$ in $\Cc$ has a factorization $f= m \circ q$, with $q$ a regular epimorphism and $m$ a monomorphism; 
\item this factorization is unique (up to isomorphism).
\end{enumerate}
\end{theorem}
\begin{proof}
\begin{enumerate}
\item
Let $f \colon A \rightarrow B$ be an arrow in $\Cc$. Consider the diagram here below where
 $(\Eq(f), f_{1},f_{2})$ is the kernel pair of $f$, $q$ is the coequalizer of  $(f_{1},f_{2})$, and $m$ the unique arrow such that $m\circ q=f$.
\begin{equation}\label{Factorisation-regular}
\xymatrix{\Eq(f) \ar@<2pt>[r]^{f_1}\ar@<-2pt>[r]_{f_2}_{.}&A\ar[r]^{q}\ar[dr]_f&I \ar@{..>}[d]^m\\&&B}
\end{equation}
We need to show that $m$ is a monomorphism or, equivalently, that the projections $p_1 \colon \Eq(m) \rightarrow I $ and $p_2 \colon \Eq(m)\rightarrow I $ of the kernel pair of $m$ are equal.
For this, consider the diagram
$$\xymatrix{ \Eq(f) \ar[rr]^-{b}
          \ar[dd]_{a}&&
 \Eq(m) \times_I A \ar[rr]^-{\pi_2}\ar[dd]_{\pi_1}&& A\ar[dd]^{q}\\&&\\
 A\times_{I} \Eq(m) \ar[rr]_-{\phi_{2}}\ar[dd]_{\phi_{1}}&&\Eq(m) \ar[rr]_-{p_{2}}\ar[dd]_{p_{1}} &&
 I \ar[dd]^{m}\\&&\\A\ar[rr]_{q}&&I \ar[rr]_{m}&&B}$$\\
where all the squares are pullbacks. We know that the whole square is then a pullback, so that one can assume that $f_1= \phi_1 \circ a$ and $f_2= \pi_2 \circ b$.
The arrow $\phi_2 \circ a = \pi_1 \circ b$ is then an epimorphism, as a composite of epimorphisms (we have used the pullback stability of regular epimorphisms). The fact that $\phi_1 \circ a = f_1$ and  $\pi_2 \circ b = f_2$ implies that 
$$p_1 \circ( \phi_2 \circ a)= q \circ  \phi_1 \circ a  = q \circ f_1 = q \circ f_2 = q \circ \pi_2 \circ b = p_2 \circ \pi_1 \circ b =  p_2 \circ  (\phi _2 \circ a).$$
Since $\phi _2 \circ a$ is an epimorphism, it follows that $p_1 = p_2$, as desired.
 \item To prove the uniqueness of the factorization one can use the fact that any regular epimorphism is a strong epimorphism.
 \end{enumerate}
\end{proof}

\begin{rem}
The uniqueness of the regular image of any arrow $f$ in Theorem \ref{theo1} allows one to call the subobject $m \colon I \rightarrow B$ in diagram \eqref{Factorisation-regular} 
\emph{the image of $f$}.
\end{rem}

\begin{prop}\label{prop1}
Let $\Cc$ be a regular category, then the following properties are satisfied:
\begin{enumerate}
\item regular epimorphisms coincide with strong epimorphisms;
\item if $g\circ f$ is a regular epimorphism, then $g$ is a regular epimorphism;
\item if $g$ and $f$ are regular epimorphisms, then $g\circ f$ is a regular epimorphism;
\item if $f \colon X \rightarrow Y$ and $g \colon X' \rightarrow Y'$ are regular epimorphisms, then the induced arrow $f\times g \colon X \times X' \rightarrow Y \times Y'$ is also a regular epimorphism.
\end{enumerate}
\end{prop}
\begin{proof}
\begin{enumerate}
\item 
One needs to check that any strong epimorphism $f \colon {A \to B}$ is a regular epimorphism.
Consider the factorization $f=m \circ q$ of a strong epimorphism, with $m$ a monomorphism and $q$ a regular epimorphism (Theorem \ref{theo1}). The commutativity of the diagram
$$\xymatrix{A\ar[rr]
          ^{f}\ar[d]
           _{q}&&B\ar@{=}[d]^{}\ar@{.>}[dll]_{d}\\
 I \ar[rr]_{m}&&B}$$
yields a unique arrow $d\colon B \to I$ such that $d \circ  f=q$
and $m \circ d=1_{B}$. This arrow $d$ is the inverse of $m$, and $f$ is then a regular epimorphism.
\item Follows from (1) and the properties of strong epimorphisms (Exercices \ref{exo-strong}).
\item Same argument as for $(2).$
\item If $f\colon X\to Y$ is a regular epimorphism,
consider the commutative diagram
$$\xymatrix{X\times X'\ar[rr]^{f\times 1_{X'}}\ar[d]_{\pi_1}&&Y\times X'\ar[d]^{\pi_1} \\ X\ar@{->>}[rr]_{f}&&Y}$$
 which is easily seen to be a pullback.
The arrow $f\times 1_{X'} $ is then a regular epimorphism and, similarly, one checks that $1_Y\times g $ is a regular epimorphism.
Since $f \times g = (1_Y\times g)\circ  (f\times 1_{X'})$,  this arrow is a regular epimorphism by $(3)$.  
\end{enumerate}
\end{proof}
We are now going to give an equivalent formulation of the notion of regular category:

\begin{theorem}\label{theo2}
Let $\Cc$ be a finitely complete category. Then $\Cc$ is a regular category if and only if
\begin{enumerate}
\item any arrow in $\Cc$ factorizes as a regular epimorphism followed by a monomorphism;
\item these factorizations are \emph{pullback-stable}:  if $m\circ q$ is the factorization of an arrow $p \colon E \rightarrow B$, $f \colon A \rightarrow B$ any arrow, and the squares
$$
\xymatrix{E \times_B A \ar[r]^-{q'} \ar[d]_{\pi_1} & E' \times_B A \ar[r]^-{m'} \ar[d] & A \ar[d]^f \\
E \ar@{->>}[r]_-{q} &{E'\, \, }  \ar@{>->}[r]_{m} & B}
$$
are pullbacks, then $m' \circ q'$ is the factorization of the pullback projection $\pi_2 \colon E \times_B A \rightarrow A$.
\end{enumerate}
\end{theorem}
\begin{proof} When $\mathbb C$ is regular, $(1)$ and $(2)$ follow from Theorem \ref{theo1}.\\ For the converse, it is clear that $(2)$ implies that regular epimorphisms are pullback stable. It remains to show that any kernel pair \begin{equation}\label{kernel-p} \xymatrix{\Eq(f) \ar@<4pt>[r]^{f_1}\ar@<-4pt>[r]_{f_2}&X}\end{equation}
 of an arrow $f \colon X\rightarrow Y$ has a coequalizer. For this consider the regular epi-mono factorization $m\circ q$ of $f$ (which exists by $(1)$), and observe that \eqref{kernel-p}
  is also the kernel pair of the regular epimorphism $q$, since $m$ is a monomorphism. The arrow $q$ is then the coequalizer of its kernel pair \eqref{kernel-p}  (see Exercise \ref{pairenoyau}).
\end{proof}

The following result will be useful to prove the so-called Barr-Kock Theorem:
\begin{lemma}\label{lefttoright}
Consider a commutative diagram
$$
\xymatrix{A\ar[r]^k \ar[d]_a &B \ar[d]^b\ar[r]^l & C \ar[d]_{c}  \\
A' \ar[r]_{k'} &B' \ar[r]_{l'} & C'}
$$
in a regular category $\mathbb C$, 
where the left-hand square and the external rectangle are pullbacks. If $k'$ is a regular epimorphism, then the right-hand square is a pullback.
\end{lemma}

\begin{proof}
Consider the commutative diagram
\[
\xymatrix{A \ar[rr]^{k}\ar[dd]_{a}\ar@{.>}[dr]^{\alpha} && B \ar[dd]_(.3){b}|(.5){\hole}
\ar[rr]^{l}\ar@{.>}[dr]^{\beta}&&C \ar[dd]_(.3){c}|(.5){\hole} \ar@{=}[dr] \\
& A' \times_{C'}{C} \ar[dl]^-{\pi_1}\ar[rr]_(.35){\pi_2}&&B' \times_{C'}{C} \ar[rr]_(.35){\pi_{2}'}\ar[dl]^-{\pi_{1}'}&&C \ar[dl]^{c}\\ 
A' \ar[rr]_{k'} && B' \ar[rr]_{l'}&&C'}
\]
where $(B' \times_{C'}{C}, \pi_1',\pi_2')$ is the pullback of $l'$ and $c$, and $(A' \times_{C'}{C}, \pi_1, \pi_2)$ is the pullback of $k'$ and $\pi_1'$, with $\alpha$ and $\beta$ the naturally induced arrows. The fact that the external rectangle is a pullback implies that the arrow $\alpha$ is an isomorphism. The arrow $\pi_2$ is a regular epimorphism (because $k'$ is one), so that $\pi_2 \circ \alpha= \beta \circ k$ is a regular epimorphism, and then $\beta$ is a regular epimorphism (see Proposition~\ref{prop1} (2)). The arrow $\beta$ is a monomorphism: this follows from the fact that the square
$$
\xymatrix{A \ar[d]_{\alpha} \ar[r]^k & B \ar[d]^{\beta} \\
A' \times_{C'}{C} \ar[r]_{\pi_2} & B' \times_{C'}{C}
}
$$
is a pullback, so that both the induced commutative squares 
$$
\xymatrix{\Eq(\alpha) \ar@<-.5ex>[d]_{p_1} \ar@<.5ex>[d]^{p_2} \ar@{.>}[r] & \Eq(\beta) \ar@<-.5ex>[d]_{p_1} \ar@<.5ex>[d]^{p_2} \\
A \ar[r]_{k} & B
}
$$
 are pullbacks, where the (unique) dotted arrow making them commute is then a (regular) epimorphism. The arrows $p_1 \colon \Eq(\alpha) \rightarrow A$ and $p_2 \colon \Eq(\alpha) \rightarrow A$ are equal (since $\alpha$ is an monomorphism), so that the projections $p_1 \colon \Eq(\beta) \to B$ and $p_2 \colon \Eq(\beta) \to B$ are also equal, and then $\beta$ is indeed a monomorphism.
\end{proof}

We are now ready to prove the following interesting result, often referred to as the \emph{Barr-Kock Theorem} \cite{Barr}, although it was first observed by A. Grothendieck \cite{Gro} in a different context (see also \cite{BG}): 
\begin{theorem} \label{BK}
Let $\Cc$ be a regular category, and 
$$\xymatrix{\Eq(f) \ar@<4pt>[rr]^{p_1}\ar[dd]_{v}\ar@<-4pt>[rr]_{p_2}&&A\ar[dd]^{u}
\ar@{->>}[rr]^{f} &&X\ar[dd]_{w}\\&(1)&&(2)\\ \Eq(g) \ar@<4pt>[rr]^{p_1}
\ar@<-4pt>[rr]_{p_2}&&B\ar[rr]_{g} &&Y}$$\\
a commutative diagram with $f$ a regular epimorphism. If either of the left-hand commutative squares are pullbacks, then the right-hand square $(2)$ is a pullback.
\end{theorem}
\begin{proof}
Consider the following commutative diagram
$$\xymatrix{\Eq(f) \ar[rr]^{p_2} \ar[dd]_{v}\ar[dr]_{p_1}&&A \ar[dd]^(.7){u}\ar[dr]^{f}\\& A\ar[rr]_(.7){f}
\ar[dd]^(.7){u}&&X \ar[dd]^{w} \\ \Eq(g) \ar[rr]^(.7){p_2}\ar[dr]_{p_1}&&B\ar[dr]_{g}\\&B \ar[rr]_{g}&&Y}$$
The assumptions guarantee that the left-hand face and the bottom face of the cube are pullbacks. By commutativity it follows that the rectangle
$$
\xymatrix{ \Eq(f) \ar[r]^{p_2} \ar[d]_{p_1} &A \ar[d]^f \ar[r]^u & B \ar[d]^{g}  \\
A \ar@{->>}[r]_{f} &X  \ar[r]_{w} & Y}
$$
is also a pullback, as well as its left-hand square. Since $f$ is a regular epimorphism, by Lemma \ref{lefttoright} it follows that the right-hand square is a pullback. 
\end{proof}

\section{Relations in regular categories}
\begin{defi}\label{def8} A \emph{relation 
from  $X$ to $Y$} in a category $\Cc$ is a graph  $$ \xymatrix{ & R \ar[rd]^{r_{2}} \ar[dl]_{r_{1}} & \\ 
X& &Y }$$
such that the pair $(r_{1},r_{2})$ is jointly monomorphic. When the product $X \times Y$ exists, this is equivalent to the fact that the factorization $(r_{1},r_{2}):R\rightarrow X\times Y$ is a monomorphism.
\end{defi}
 As usual, we identify two relations $R \rightarrow X\times Y$ and $S \rightarrow X\times Y$ when they determine the same \emph{subobject} of $X \times Y$, i.e. the same equivalence class of monomorphisms with codomain $X \times Y$. If $X=Y$, one says that $R$ is a relation on $X$.

\begin{enumerate}
\item A relation $R$ on $X$ is \emph{reflexive} when there is an arrow $\delta :X\rightarrow R$
such that $r_{1}\circ \delta=1_{X}=r_{2}\circ \delta$.
\item $R$ is \emph{symmetric} if there is an arrow $\sigma :R\rightarrow R$ such that $r_{1}\circ \sigma=r_{2}$
and $r_{2}\circ \sigma=r_{1}$.
\item Consider the pullback
$$\xymatrix{R\times_{X} R\ar[r]^-{p_{2}}\ar[d]_{p_{1}}&R\ar[d]^{r_{1}}\\
R\ar[r]_{r_{2}}&X.}$$
\noindent The relation $R$ is \emph{transitive} if there is an arrow $\tau : R\times_{X} R\rightarrow R$ such that $r_{1}\circ \tau=r_{1}\circ p_{1}$
et $r_{2}\circ \tau=r_{2}\circ p_{2}$.
\end{enumerate}
A relation $R$ on $X$ is an \emph{equivalence relation} if $R$ is reflexive, symmetric and transitive.
Of course, this abstract notion of equivalence relation gives in particular the usual one when $\mathbb C$ is the category of sets. 

When $\mathbb C = \mathsf{Grp}$ is the category of groups, an equivalence relation $R\subset X \times X$ in the category $\Grp$ is an equivalence relation on the underlying set of $X$ which is also a subgroup of the group $X \times X$. In universal algebra, an internal equivalence relation in a variety is called a \emph{congruence}.

\begin{lemma}\label{lem8} In a category with pullbacks the kernel pair $\xymatrix{\Eq(f) \ar@<4pt>[r]^-{p_{1}}\ar@<-4pt>[r]_-{p_{2}}&X}$
of an arrow $f:X\rightarrow Y$ is an equivalence relation on $X$ in $\Cc$.
\end{lemma}
\begin{proof}
The arrows $p_1 \colon \Eq(f) \rightarrow X$ and $p_2 \colon \Eq(f) \rightarrow X$ are jointly monomorphic, since they are projections of a pullback.
The universal property of the kernel pair $(\Eq(f), p_1,p_2)$ implies that there is a unique $\delta \colon X \rightarrow \Eq(f)$ such that $p_1 \circ \delta = 1_X= p_2 \circ \delta$
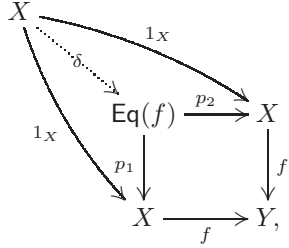
 \[\xymatrix{ X \ar@/_/[ddr]_{1_X  } \ar@/^/[drr]^{1_X } \ar@{.>}[dr]|-{\delta  }\\
& \Eq(f) \ar[d]_{p_1} \ar[r]^{p_2} & X \ar[d]^f \\ &X \ar[r]_f &Y,
}\]
and $\Eq(f)$ is then reflexive. Similarly, the commutativity of the external part of the diagram 
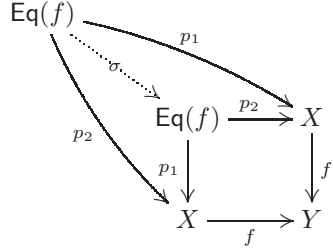
\[\xymatrix{ \Eq(f) \ar@/_/[ddr]_{p_2  } \ar@/^/[drr]^{p_1 } \ar@{.>}[dr]|-{\sigma  }\\
& \Eq(f) \ar[d]_{p_1} \ar[r]^{p_2} & X \ar[d]^f \\ &X \ar[r]_f &Y
}\]
implies that there is a unique arrow $\sigma \colon \Eq(f) \rightarrow  \Eq(f)$ such that $p_1 \circ \sigma = p_2$ and $p_2 \circ \sigma = p_1$, hence $ \Eq(f)$ is symmetric. For the transitivity of $ \Eq(f)$ one considers the following commutative diagram
$$
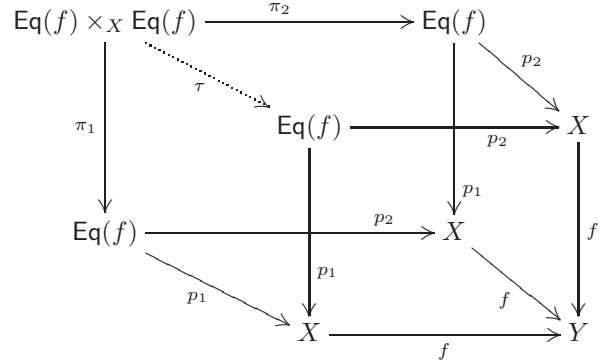
\xymatrix{ \Eq(f) \times_X  \Eq(f) \ar[rr]^{\pi_2} \ar[dd]_{\pi_1 }\ar@{.>}[dr]_{\tau}&& \Eq(f)  \ar[dd]^(.8){p_1}\ar[dr]^{p_2}\\&  \Eq(f) \ar[rr]_(.7){p_2}
\ar[dd]^(.7){p_1}&&X \ar[dd]^{f} \\ \Eq(f) \ar[rr]^(.8){p_2}\ar[dr]_{p_1}&&X \ar[dr]_{f}\\&X \ar[rr]_{f}&&Y}$$
where the back face is a pullback. The universal property of the kernel pair $( \Eq(f), p_1, p_2)$ shows that there is a unique $\tau$ such that $p_1 \circ \tau =p_1 \circ \pi_1$ and $p_2 \circ \tau = p_2 \circ \pi_2$.
\end{proof}

An important aspect of regular categories is that in these categories one can define a composition of relations,  which has some nice properties. \\ 

In the category $\Ens$, if $R \rightarrow X \times Y$ is a relation from $X$ to $Y$ and $S  \rightarrow Y \times Z$ a relation from $Y$ to $Z$, one usually defines the relation $ S \circ R \rightarrow X \times Z$ by setting
$$S \circ R = \{  (x,z) \in X \times Z \,  {\rm such \, that} \,  \exists \, y \in Y \,  {\rm with} \, xRy, ySz \}.$$
This construction is also possible in any regular category $ \mathbb C$, thanks to the existence of regular images (Theorem \ref{theo1}).
One first builds the pullback
$$ 
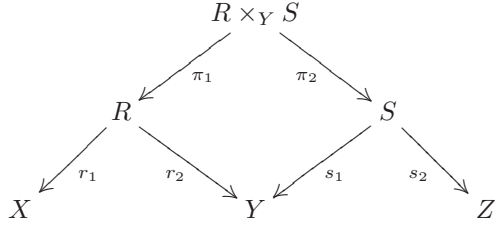
\xymatrix{ & & R \times_Y S   \ar[rd]_{\pi_{2}} \ar[dl]^{\pi_{1}} & & \\ 
& R\ar[dl]^{r_1} \ar[dr]_{r_2}& &S \ar[dl]^{s_1} \ar[dr]_{s_2}& \\
X & & Y & &  Z}$$
and one then factorizes the arrow $(r_1 \circ \pi_1, r_2 \circ \pi_2) \colon R \times_Y S \rightarrow X \times Z$ as a regular epimorphism
  $q \colon R \times_Y S  \rightarrow I$  followed by a monomorphism $i \colon I \rightarrow X \times Z$:
$$ \xymatrix{
 R \times_Y S  \ar[r]^-q & I \ar[r]^-i & X \times Z
 }$$
 In $\Ens$, the set $I$ consists of the element $(x,z) \in X \times Z$ such that there is a $(u,y,v) \in R \times_Y S$ with $u=x$ and $v=z$: this is precisely $S \circ R$.
 
 This composition is actually associative:
 \begin{theorem}\label{associative}
Let $\mathbb C$ be a regular category. If $R \rightarrow A \times B$, $S \rightarrow B \times C$ and $T  \rightarrow C \times D$ are relations in $\mathbb C$, one has the equality 
  $$T \circ (S\circ R)= (T \circ S) \circ R.$$
 \end{theorem}
\begin{proof}
Consider the diagram obtained by building the following pullbacks:
$$
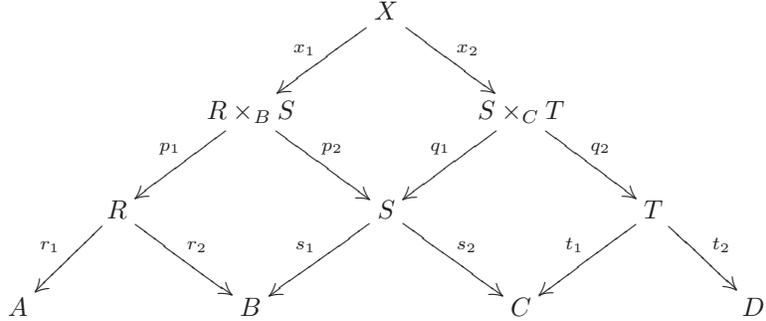
\xymatrix{ & & & X \ar[dr]^{x_2} \ar[dl]_{x_1}  & & & \\
& &R \times _B S \ar[dr]^{p_2} \ar[dl]_{p_1} & & S \times _C T  \ar[dr]^{q_2} \ar[dl]_{q_1}& & \\
&R \ar[dr]^{r_2} \ar[dl]_{r_1} & &  S\ar[dr]^{s_2} \ar[dl]_{s_1}& & T \ar[dr]^{t_2} \ar[dl]_{t_1} & \\
A & & B & & C & &D. }
$$
The proof consists in showing that the relations $T \circ (S\circ R)$ and $(T \circ S)\circ R$ are both given by the regular image $i \colon I \rightarrow A \times D$ in the factorization 
$$\xymatrix{X \ar[rrrr]^-{(r_1 \circ p_1 \circ x_1, t_2 \circ q_2 \circ x_2 )}  \ar@{->>}[drr]_q& & & & A \times D \\
& & I \ar[urr]_i &  & }$$
as a regular epimorphism followed by a monomorphism of the arrow $$(r_1 \circ p_1 \circ x_1, t_2 \circ q_2 \circ x_2 ) \colon X \rightarrow A \times D.$$ We leave it to the reader the verification of this fact, which uses the pullback stability of regular epimorphisms in a crucial way. 
\end{proof}

This result allows one to define a new category starting from any regular category $\mathbb C$, the category $\mathsf{Rel} (\mathbb C)$ of relations in $\mathbb C$. The objects are the same as the ones in $\mathbb C$, an arrow from $X$ to $Y$ is simply a relation from $X$ to $Y$, and composition is the relational one defined above.
For any relation $R$ from $X$ to $Y$ the discrete relation  (also called the equality relation) on $X$ $$\Delta_X : \xymatrix{X \ar@<2pt>[r]^{1_X} \ar@<-2pt>[r]_{1_X} & X}$$
 is such that $R \circ \Delta_X=R$, and for any relation $S$ from $Z$ to $X$ one has $\Delta_X \circ S=S$. It follows that the arrow $\Delta_X$ in $\mathsf{Rel}(\mathbb C)$ is the identity on the object $X$ for the composition in $\mathsf{Rel} (\mathbb C)$.
 
There is a faithful functor $ \Gamma \colon \mathbb C \rightarrow \mathsf{Rel} (\mathbb C)$, where $\Gamma(f)$ is the 
 the \emph{graph} of $f \colon X \rightarrow Y$, seen as a relation:
$$
\xymatrix{&X \ar[dl]_{1_X} \ar[dr]^{f} & \\
X& & Y.
}
$$
From now on we shall write $1_X$ for the discrete relation on $X$, which can also be seen as the relation $\Gamma(1_X)$.
\begin{rem}
$\mathsf{Rel} (\mathbb C)$ is not only a category, but a (locally ordered) $2$-category. Indeed, there is a natural partial ordering on its arrows, since the relations from $X$ to $Y$ are the subobjects of a fixed object $X \times Y$ of $\mathbb C$. This order is also compatible with the composition: if $R\le S$, then $R \circ T \le S \circ T$ whenever these composites are defined. This is the main argument to show that $\mathsf{Rel} (\mathbb C)$ is a $2$-category, which is actually locally-ordered: between any two arrows (or $1$-cells) there is at most one $2$-cell, and the only invertible $2$-cells are the identities (see \cite{Johnstone-elephant} for more details).
\end{rem}
\section{Calculus of relations and Mal'tsev categories}
In this section we shall always assume that the category $\Cc$ is regular.\\
Given a relation $R= (R, r_1, r_2)$
$$
\xymatrix{&R \ar[dl]_{r_1} \ar[dr]^{r_2} & \\
X& & Y
}
$$
 from $X$ to $Y$, we write $R^o= (R, r_2, r_1)$ for the \emph{opposite relation} from $Y$ to $X$:
$$
\xymatrix{&R \ar[dl]_{r_2} \ar[dr]^{r_1} & \\
Y& & X.
}
$$
Of course ${(R^o)}^o= R$. It is easy to see that a relation $R$ is \emph{symmetric} if and only if $R=R^o$.
On the other hand, a relation $R$ is \emph{transitive} when $R \circ R \le R.$
Moreover, in a regular category, any relation $(R,r_1,r_2)$ can be seen as the composite $R= r_2 \circ r_1^o$. By definition of the composition of relations, the relation $(X \times_Y Z, p_1, p_2)$ in a pullback 
$$\xymatrix{X \times_Y Z \ar[r]^-{p_2}\ar[d]_{p_1}&Z \ar[d]^{g}\\
X\ar[r]_{f}&Y}$$
 can be written as $g^o \circ f$.
We leave the verification of the following properties to the reader:
\begin{lemma}\label{simple-props}
In a regular category $\Cc$:
\begin{enumerate}
\item any kernel pair $(\Eq (f), f_1, f_2)$ of an arrow $f \colon A \rightarrow B$ can be written as $f^o \circ f$;
\item $f \colon A \rightarrow B$  is a regular epimorphism if and only if $f \circ f^o= 1_B$;
\item $f \colon A \rightarrow B$  is a monomorphism if and only if $f^o \circ f = 1_A$.
\end{enumerate}
\end{lemma}
The relations that are ``maps'', i.e. of the form \begin{equation}\label{map}
\xymatrix{&X \ar[dl]_{1_X} \ar[dr]^{f} & \\
X& & Y,
}
\end{equation}
for some arrow $f$ in $\Cc$, have the following additional property: 
\begin{lemma}\label{difunctional}
Any relation of the form (\ref{map})
is \emph{difunctional:} $$f \circ f^o \circ f = f.$$
\end{lemma}
\begin{proof}
The relation $f \circ f^o \circ f = f$ is obtained as the regular image of the external graph in the following diagram,

$$\xymatrix{ & & & \Eq(f) \ar[dl]_{1_{\Eq(f)}} \ar[dr]^-{p_2} & & & \\
 & &\Eq(f) \ar[dl]_-{p_1} \ar[dr]^-{p_2}& &{ \quad A \quad }\ar[dl]_{1_A}\ar[dr]^{1_A} & & \\
 &A \ar[dl]_{1_A} \ar[dr]^f & &A \ar[dl]_f  \ar[dr]^{1_A}   & & A \ar[dr]^f \ar[dl]_{1_A} & \\
 A & & B && A& & B,
}$$
which is simply the regular image of the graph 
$$
\xymatrix{&\Eq(f) \ar[dl]_{p_1} \ar[dr]^{f \circ p_2} & \\
A& & B.
}
$$
Since $p_1 \colon \Eq(f) \rightarrow A$ is a split epimorphism, thus in particular a regular epimorphism (by Proposition \ref{epis}), we see that the relation $f \circ f^o \circ f$ is given by the relation $(1_A, f)$ in the commutative diagram
$$
\xymatrix{&\Eq(f) \ar[dl]_{p_1} \ar@{.>}[d]^{p_1} \ar[dr]^{f \circ p_2} & \\
A& A \ar[l]^{1_A}  \ar[r]_f & B,
}
$$
as desired.
\end{proof}
In the category of sets the notion of \emph{difunctional relation} was first introduced by Riguet \cite{Riguet}. A relation $R$ is difunctional if the fact that $(x,y) \in R, (z,y) \in R$ and $(z,u) \in R$ implies that $(x,u)\in R$. This property can be expressed in any regular category as follows:
\begin{defi}
A relation $(R, r_1, r_2)$ from $X$ to $Y$ in a regular category is \emph{difunctional} if $$R \circ R^o \circ R = R.$$
\end{defi}
The following notion was introduced by A. Carboni, J. Lambek and M.C. Pedicchio in \cite{CLP}, and it has been investigated in several articles in the last 30 years.
\begin{defi}
\emph{A finitely complete category $\Cc$ is called a \emph{Mal'tsev} category if any internal reflexive relation in $\Cc$ is an equivalence relation.}
\end{defi}
The following characterization of regular Mal'tsev categories can be found in \cite{CLP} (see also \cite{Meisen}). It is a good example of a proof using the so-called calculus of relations.
\begin{theorem}\label{Maltsev}
Let $\Cc$ be a regular category. Then the following conditions are equivalent:
\begin{enumerate}
\item for any pair of equivalence relations $R$ and $S$ on any object $X$ in $\Cc$, \\ $S \circ R$ is an equivalence relation;
\item for any pair of equivalence relations $R$ and $S$ on any object $X$ in $\Cc$, \\ $S \circ R = R \circ S$;
\item for any pair of kernel pairs $\Eq(f)$ and $\Eq(g)$ on any object $X$ in $\Cc$, \\ $\Eq(g)  \circ \Eq(f) = \Eq(f) \circ \Eq(g)$;
\item any relation $U$ from $X$ to $Y$ in $\mathbb C$ is difunctional; 
\item any reflexive relation $R$ on an object $X$ in $\Cc$ is an equivalence relation;
\item any reflexive relation $R$ on an object $X$ in $\Cc$ is symmetric;
\item any reflexive relation $R$ on an object $X$ in $\Cc$ is transitive.
\end{enumerate}
\end{theorem}
\begin{proof}
$(1) \Rightarrow (2)$ By assumption the relation $S \circ R$ is an equivalence relation, thus it is symmetric:
$$(S \circ R)^o =S \circ R.$$
Since both $S$ and $R$ are symmetric it follows that $$R \circ S = R^o \circ S^o = (S \circ R)^o = S \circ R.$$
$(2) \Rightarrow (3)$ Obvious, since any kernel pair is an equivalence relation (Lemma \ref{lem8}). \\
$(3) \Rightarrow (4)$ Any relation $(U, u_1, u_2)$ can be written as $U = u_2 \circ u_1^o$. The assumption implies that the kernel pairs $\Eq(u_1)$ and $\Eq(u_2)$ of the projections commute in the sense of the composition of relations (on the object $U$):
$$(u_2^o \circ u_2) \circ  (u_1^o \circ u_1)= (u_1^o \circ u_1) \circ  (u_2^o \circ u_2).$$ Accordingly, by keeping in mind that the relations $u_1$ and $u_2$ are difunctional (by Lemma \ref{difunctional}) and $\Eq(u_1) = u_1^o \circ u_1$ and $\Eq(u_2) = u_2^o \circ u_2$, one has:
\begin{eqnarray}
U  & =& u_2 \circ u_1^o \nonumber  \\
 & =& (u_2 \circ u_2^o \circ u_2) \circ (u_1^o \circ u_1 \circ u_1^o ) \nonumber  \\
& = & u_2 \circ (u_2^o \circ u_2) \circ (u_1^o \circ u_1) \circ u_1^o \nonumber \\
& = & u_2 \circ (u_1^o \circ u_1) \circ (u_2^o \circ u_2) \circ u_1^o \nonumber \\
& = & (u_2 \circ u_1^o) \circ (u_1 \circ u_2^o) \circ (u_2 \circ u_1^o) \nonumber \\
& = & U \circ U^o \circ U. \nonumber 
\end{eqnarray}
$(4) \Rightarrow (5)$ Let $(U,u_1, u_2)$ be a reflexive relation on an object $X$, so that $1_X \le U$. 
 By difunctionality we have: $$U^o = 1_X \circ U^o \circ 1_X \le U \circ U^o \circ U = U,$$ showing that $U$ is symmetric. On the other hand:
$$ U \circ U = U \circ 1_X \circ U \le U \circ U^o \circ U = U,$$ and $U$ is transitive.

$(5) \Rightarrow (6)$ Clear. \\
$(6) \Rightarrow (1)$
 First observe that $S \circ R$ is reflexive, since both $S$ and $R$ are reflexive:
$$ 1_X = 1_X \circ 1_X \le S \circ R.$$
By assumption the relation $S \circ R$ is then symmetric, so that $$R \circ S = R^o \circ S^o = (S \circ R)^o = S \circ R.$$
The relation $S \circ R$ is transitive:
$$S \circ R = (S \circ S) \circ (R \circ R)=  S \circ (S \circ R) \circ R= S \circ (R \circ S) \circ R = S \circ R \circ S \circ R.$$

Observe that $(5) \Rightarrow (7)$ is obvious, and let us prove that $(7) \Rightarrow (4)$.
Let $U= u_2 \circ u_1^o$ be any relation from $X$ to $Y$.
The relation $$u_2^o \circ u_2\circ  u_1^o \circ u_1$$ is reflexive, thus it is transitive by assumption. This gives the equality
$$(u_2^o \circ u_2\circ  u_1^o \circ u_1) \circ (u_2^o \circ u_2\circ  u_1^o \circ u_1) = u_2^o \circ u_2\circ  u_1^o \circ u_1, $$ yielding 
$$u_2 \circ u_2^o \circ u_2\circ  u_1^o \circ u_1 \circ u_2^o \circ u_2\circ  u_1^o \circ u_1 \circ u_1^o= u_2 \circ u_2^o \circ u_2\circ  u_1^o \circ u_1 \circ u_1^o.$$
By difunctionality of $u_2$ and $u_1^o$ we conclude that
$$ u_2 \circ  u_1^o \circ u_1 \circ u_2^o \circ u_2\circ u_1^o= u_2 \circ u_1^o,$$
and $$U \circ U^o \circ U =U.$$
\end{proof}
\begin{exs}\label{exemples-M}
The categories $\Grp$, $\Ab$, $R$-$\mathsf{Mod}$, $\mathsf{Rng}$ and $\mathsf{Grp(Top)}$ are all Mal'tsev categories. By Theorem \ref{Maltsev} (6), to see this it suffices to show that any (internal) reflexive relation $R$ on any object $X$ in these categories is symmetric.
Let us check this property for the category $\mathsf{Grp}$ of groups: given an element $(x,y)$ of a reflexive relation $R$ which is also a subgroup of $X \times X$, we know that its inverse $(x^{-1}, y^{-1})$ is also in $R$ and, by reflexivity, both $(x,x)$ and $(y,y)$ belong to $R$. It follows that 
$$(x,x) \cdot (x^{-1}, y^{-1}) \cdot (y,y) = (x \cdot x^{-1} \cdot y, x \cdot y^{-1}\cdot y)= (y,x) \in R$$  and $\mathsf{Grp}$ is a Mal'tsev category. 
An inspection of the proof for $\mathsf{Grp}$ shows that the argument is still valid if the theory of an algebraic variety has a term $p(x,y,z)$ such that $p(x,y,y)=x$ and $p(x,x,y)= y$. Varieties of algebras having such a ternary term $p$ are called Mal'tsev varieties  \cite{Smith}, or $2$-permutable varieties, and the term $p$ a Mal'tsev operation. This terminology is motivated by the famous Mal'tsev theorem asserting that a variety $\mathbb V$ of algebras has the property that each pair $R$ and $S$ of congruences on an algebra $A$ in $\mathbb V$ permute, i.e. $R \circ S = S \circ R$ if and only if its theory has a ternary Mal'tsev operation \cite{Mal}.

Of course, any variety of algebras whose theory contains the operations and identities of the theory of groups is a Mal'tsev variety.

For a different example, consider the variety $\mathsf{QGrp}$ of quasigroups \cite{Smith}: its algebraic theory has a multiplication $\cdot$, a left division $\backslash$ and a right division $/$ such that $x\backslash (x \cdot y) = y$, $(x \cdot y) /y =x$, $x\cdot (x \backslash y)=y$ and $(x/y)\cdot y = x$. \\ A Mal'tsev operation for the theory of quasigroups is given by the term $$p(x,y,z) = (x/(y\backslash y))\cdot (y\backslash z),$$
since $$p(x,y,y) = (x/(y\backslash y))\cdot (y\backslash y) = x,$$
and $$p(x,x,y) = (x/(x\backslash x))\cdot (x\backslash y) = ((x \cdot (x \backslash x)/(x\backslash x))\cdot (x\backslash y)) = x \cdot (x\backslash y) = y.$$
The category $\mathsf{Heyt}$ of Heyting algebras is a Mal'tsev variety \cite{Johnstone}, with a Mal'tsev operation defined by the term
$$p(x,y,z) = ((x\rightarrow y ) \rightarrow z) \wedge ((z \rightarrow y) \rightarrow x)). $$
For the axioms and basic properties of Heyting algebras we refer the reader to \cite{Johnstone}, or to the Chapter \emph{Notes on point-free topology} \cite{PP} in this volume. One observes that \begin{eqnarray}
p(x,x,y) & = & ((x\rightarrow x ) \rightarrow y) \wedge ((y \rightarrow x) \rightarrow x)) \nonumber  \\
& = & (1 \rightarrow y) \wedge ((y \rightarrow x) \rightarrow x)) \nonumber \\
& =& y \wedge ((y \rightarrow x) \rightarrow x))  \nonumber \\
& = & y \nonumber
\end{eqnarray}
and 
 \begin{eqnarray}
 p(x,y,y) & = & ((x\rightarrow y ) \rightarrow y) \wedge ((y \rightarrow y) \rightarrow x)) \nonumber  \\
 & = & ((x\rightarrow y ) \rightarrow y) \wedge (1 \rightarrow x) \nonumber \\
 & = & ((x\rightarrow y ) \rightarrow y) \wedge x \nonumber \\
 & = & x. \nonumber
 \end{eqnarray}
Other examples of regular Mal'tsev categories are: any regular additive category, therefore in particular any abelian category \cite{Buch}, and the dual of any elementary topos \cite{CKP}. 
The category of ${\mathbb C}^*$-algebras and the
category $\mathsf{Hopf}_{K, coc}$ of cocommutative Hopf algebras over a field $K$ are also regular Mal'tsev categories \cite{GRos,GSV}.

On the other hand, the categories $\mathsf{Set}$ of sets and $\mathsf{Mon}$ of monoids are regular categories which are not Mal'tsev ones. Indeed, the usual order relation $\le$ on $\mathbb N$ is an internal reflexive relation (both in $\mathsf{Set}$ and in $\mathsf{Mon}$) which is not symmetric.
\end{exs}
An important property of regular Mal'tsev categories is expressed in terms of diagrams of the form
\begin{equation}\label{Regular-pushout}
	\xymatrix{
		C  \ar@{->>}[r]^{c} \ar@<-2pt>[d]_{g} & A \ar@<-2pt>[d]_f\\
		D   \ar@{->>}[r]_d \ar@<-2pt>[u]_{t} & B  \ar@<-2pt>[u]_s }
\end{equation}
where $d \circ g = f \circ c$, $c \circ t = s \circ d$, $g \circ t = 1_D$, $f \circ s = 1_B$, $c$ and $d$ are regular epimorphisms. As observed in \cite{Gran} such a square is always a pushout. The following result is due to Bourn (see also \cite{CKP}): here we give an alternative proof using the calculus of relations as in \cite{GR2}:
\begin{prop}\cite{Bourn}
A regular category $\mathbb C$ is a Mal'tsev category if and only if any pushout of the form \eqref{Regular-pushout} has the property that the canonical morphism $(g,c) \colon C \rightarrow D \times_B A$ to the pullback of $d$ and $f$ is a regular epimorphism.
\end{prop}
\begin{proof}
The relation $(D \times_B A, p_1, p_2)$ which is the pullback of $d$ and $f$ can be expressed as the composite $f^o \circ d$. The regular image of $(g,c) \colon C \rightarrow D \times_B A$ is $c \circ g^o$, so that $(g,c)$ is a regular epimorphism if and only if $f^o \circ d = c \circ g^o$. Now the commutativity conditions on the square \eqref{Regular-pushout} imply that the regular image $\Eq (c)$ along $g$ is $ \Eq(d)$: $g(\Eq (c)) = \Eq(d)$.
$$\xymatrix@=30pt{\mathsf{Eq}(c) \ar@{.>>}[r] \ar@{>->}[d]_{(p_1,p_2)}  & \mathsf{Eq}(d) \ar@{>->}[d]^{(p_1,p_2)}   \\
D \times D \ar@{->>}[r]_{g \times g} & B \times B 
}
$$
 In a regular category this condition can be expressed by the equality $g \circ c^o \circ c \circ g^o = d^o \circ d$.
Since $c \circ c^o = 1_A$ by Lemma \ref{simple-props} (2), it follows that 
\begin{eqnarray}
f^o \circ d & = & c \circ c^o \circ f^o \circ d \nonumber \\
&=& c \circ g^o \circ d^o \circ d  \nonumber \\
&=& c \circ  g^o \circ (g \circ c^o \circ c \circ g^o)  \nonumber \\
& =& c \circ c^o \circ c \circ g^o \circ g \circ g^o  \nonumber\\
& =& c \circ g^o, \nonumber
\end{eqnarray}
where the fourth equality follows from the Mal'tsev assumption: $$g^o \circ g \circ c^o \circ c= \Eq(g) \circ \Eq (c) =  \Eq(c) \circ \Eq (g) = c^o \circ c \circ g^o \circ g.$$
For the converse, by Theorem \ref{Maltsev} it suffices to show that any pair of equivalence relations $\Eq(f)$ and $\Eq(g)$ which are kernel pairs of two arrows $f$ and $g$ permute. 
Note that there is no restriction in assuming that $f$ and $g$ are regular epimorphism, thanks to Theorem \ref{theo1}. 
Consider the kernel pair $(\Eq(f), f_1,f_2)$ of $f \colon X \rightarrow Y$ and the kernel pair $(\Eq(g), g_1,g_2)$ of $g \colon X \rightarrow Z$. We then consider the regular image of $\Eq(f)$ along $g$
\begin{equation}\label{direct-image}
	\xymatrix{
		{\Eq(f)}  \ar@{->>}[r]^-{\gamma} \ar@<-2pt>[d]_{f_1}  \ar@<6pt>[d]^{f_2}& g({\Eq(f)}) \ar@<-2pt>[d]_{r_1}\ar@<6pt>[d]^{r_2}\\
		X   \ar@{->>}[r]_g \ar@<-2pt>[u]_{} & Z,  \ar@<-2pt>[u] }
\end{equation}
and observe that the assumption implies that $f_2 \circ \gamma^o = g^o \circ r_2$ and $\gamma \circ f_1^o = r_1^o \circ g$. 
We then have the following identities:
\begin{eqnarray}
\Eq(f) \circ \Eq(g) & = & f_2 \circ f_1^o \circ g^o \circ g \nonumber \\
&=& f_2 \circ \gamma^o \circ r_1^o \circ g \nonumber \\
&=& g^o \circ  r_2 \circ r_1^o \circ g  \nonumber \\
& =& g^o \circ r_2 \circ \gamma \circ f_1^o \nonumber\\
& =& g^o \circ g \circ f_2 \circ f_1^o \nonumber. \\
& = & \Eq(g) \circ \Eq(f). \nonumber
\end{eqnarray}
\end{proof}
\section{Goursat categories}
In universal algebra a weaker property than the Mal'tsev axiom is the so-called $3$-permutability of congruences. Given any two congruences $R$ and $S$ on an algebra $A$ in a variety $\mathbb V$, the following equality holds:
$$R \circ S \circ R = S \circ R \circ S.$$
\begin{defi}\cite{CLP, CKP}
A regular category $\mathbb C$ is a \emph{Goursat category} if
$$R \circ S \circ R = S \circ R \circ S$$
 for any pair of equivalence relations $R$ and $S$ on any object $X$ in $\mathbb C$.
\end{defi}
\begin{exs}
\item Any regular Mal'tsev category $\mathbb C$ is a Goursat category: indeed, given any two equivalence relations $R$ and $S$ on an object $X$ in $\mathbb C$, one has:
$$R \circ (S \circ R) =  R \circ (R \circ S) = R \circ S = R \circ (S \circ S) = (S \circ R) \circ S.$$
An example of a Goursat category which is not a Mal'tsev one will be given in Example \ref{Implication}, where we shall prove that implication algebras form a Goursat variety.
\end{exs}
Among regular categories, Goursat categories are characterized by the property that equivalence relations are stable under regular images along regular epimorphisms \cite{CKP}. Here below we give a direct proof which uses the calculus of relations:
\begin{prop}\label{image}
For a regular category $\mathbb C$ the following conditions are equivalent:
\begin{enumerate}
\item $\mathbb C$ is a Goursat category;
\item for any regular epimorphism $f \colon X \rightarrow Y$ and any equivalence relation $R$ on $X$ the regular image $f(R)$ of $R$ along $f$ is an equivalence relation.
\end{enumerate}
\end{prop}
\begin{proof}
$(1) \Rightarrow (2)$. When $(R, r_1, r_2)$ is an equivalence relation it is always true that that the regular image $f(R) = f \circ R \circ f^o$ along a regular epimorphism $f \colon X \rightarrow Y$ is both reflexive and symmetric. Let us then prove that $f(R)$ is also transitive: one has the equalities 
\begin{eqnarray} f(R) \circ f(R) & = & f \circ R \circ f^o \circ f \circ R \circ f^o \nonumber \\ 
& =& f \circ (f^o \circ f ) \circ R \circ ( f^o \circ f ) \circ f^o \nonumber  \\
&=& f \circ R \circ f^o \nonumber \\
&=& f(R) \nonumber 
 \end{eqnarray}
 where the second equality follows from the Goursat assumption, and the third one from Lemma \ref{difunctional}.
 
 $(2) \Rightarrow (1)$. Conversely, consider two equivalence relations $(R, r_1, r_2)$ and $(S, s_1, s_2)$ on a same object $X$ in $\mathbb C$, and observe that the arrow $r_2 \colon R \rightarrow X$ is a split epimorphism, thus in particular a regular epimorphism. Then:
 \begin{eqnarray} R \circ S \circ R & = & (r_2 \circ r_1^o)\circ  (s_2 \circ s_1^o ) \circ (r_2 \circ r_1^o) \nonumber \\
 & = & (r_2 \circ r_1^o)\circ  (s_2 \circ s_1^o ) \circ (r_2 \circ r_1^o)^o \nonumber \\
 & = & r_2 \circ ( r_1^o\circ  s_2 \circ s_1^o  \circ r_1 ) \circ r_2^o \nonumber \\
 &=& r_2 (r_1^o \circ s_2 \circ s_1^o \circ r_1)  \nonumber \\
 & =& r_2 (r_1^{-1} (S)).  \nonumber 
 \end{eqnarray}
 Recall that the inverse image $r_1^{-1} (S)$ of the equivalence relation $S$ along $r_1$ is obtained by taking the pullback 
 $$
 \xymatrix{r_1^{-1} (S) \ar[r] \ar[d] & S \ar[d]^{(s_1,s_2)} \\
 R\times R  \ar[r]_{r_1 \times r_1} & X \times X,
 }
 $$
 and $r_1^{-1} (S)$ is always an equivalence relation. By taking into account this observation and the assumption $(2)$, one deduces that the relation $r_2 (r_1^{-1} (S)) = R \circ S \circ R$ is transitive. It follows that 
  \begin{eqnarray} 
  S \circ R \circ S & \le & R\circ S \circ R\circ S\circ R \nonumber \\
  &\le  & (R \circ S \circ R) \circ (R \circ S \circ R) \nonumber \\
  &\le  & R \circ S \circ R \nonumber 
  \end{eqnarray}
  and, symmetrically, $S \circ R \circ S \le R \circ S \circ R$, hence $S \circ R \circ S = R \circ S \circ R$.
\end{proof}
\begin{ex}
Show that the regular image of an equivalence relation in $\mathsf{Set}$ is not necessarily transitive.
\end{ex}
\begin{defi}
Consider a commutative diagram \eqref{Regular-pushout}, and the induced arrow $\hat{c}$ making the following diagram commute:
$$
\xymatrix{\Eq(g) \ar@<-.5ex>[d]_{p_1} \ar@<.5ex>[d]^{p_2} \ar@{.>}[r]^{\hat{c}}& \Eq(f) \ar@<-.5ex>[d]_{p_1} \ar@<.5ex>[d]^{p_2} \\
C \ar[r]_{c} & A
}
$$
Then the square \eqref{Regular-pushout} is called a \emph{Goursat pushout} \cite{GR} when the arrow $\hat{c}$ is a regular epimorphism.
\end{defi}
The following result was proved in \cite{GR}. Here we give a different proof of one of the two implications, based on the calculus of relations :
\begin{prop}\cite{GR}\label{G-pushout}
For a regular category $\mathbb C$ the following conditions are equivalent:
\begin{enumerate}
\item $\mathbb C$ is a Goursat category;
\item any square \eqref{Regular-pushout} is a Goursat pushout.
\end{enumerate}
\end{prop}
\begin{proof}
$(1) \Rightarrow (2)$. If $\mathbb C$ is a Goursat category then \begin{eqnarray}
c(\mathsf{Eq}(g))& = & c \circ g^o \circ g \circ c^o  \nonumber \\
& = & c \circ (c^o \circ c) \circ (g^o \circ g) \circ (c^o \circ c) \circ c^o \nonumber \\
& = & c \circ (g^o \circ g) \circ (c^o \circ c) \circ (g^o \circ g) \circ c^o \nonumber \\
& = & c \circ g^o \circ  d^o \circ d  \circ g  \circ c^o \nonumber \\
& = & c \circ c^o \circ  f^o \circ f  \circ c  \circ c^o \nonumber \\
& = &  f^o \circ f  \nonumber \\
& = &  \mathsf{Eq}( f)  \nonumber
 \end{eqnarray}
 where the third equality follows from the Goursat assumption, the fourth one from $g(\mathsf{Eq}(c) )= \mathsf{Eq}(d)$, and the sixth one from the fact that $c$ is a regular epimorphism (Lemma \ref{simple-props}).
 
$(2) \Rightarrow (1)$. Conversely, given a commutative diagram 
$$
	\xymatrix{
		{R}  \ar@{->>}[r]^-{\overline{f}} \ar@<-2pt>[d]_{r_1}  \ar@<6pt>[d]^{r_2}& f(R) = T \ar@<-2pt>[d]_{t_1}\ar@<6pt>[d]^{t_2}\\
		X   \ar@{->>}[r]_f \ar@<-2pt>[u]_{} & Y  \ar@<-2pt>[u] }
$$
where $(R, r_1,r_2)$ is an equivalence relation, $f$ is a regular epi and $(T,t_1,t_2)$ is the regular image of $R$ along $f$. We are to show that the relation $f(R)=T$ is an equivalence relation (by Proposition \ref{image}). Since the regular image of a reflexive and symmetric relation is always reflexive and symmetric, it suffices to show that $T$ is transitive.
This follows from the following computation:
\begin{eqnarray}\label{transitivity}
T \circ T & = & T \circ T^o \nonumber \\
& =& t_2 \circ t_1^o \circ t_1 \circ t_2^o \nonumber \\
& =&  t_2 \circ (\overline{f} \circ r_1^o \circ r_1 \circ \overline{f}^o) \circ t_2^o \nonumber \\
& = & f \circ r_2 \circ r_1^o \circ r_1 \circ r_2^o \circ f^o \nonumber \\
& = & f \circ R \circ R^o  \circ f^o \nonumber \\
& = & f \circ R \circ f^o \nonumber \\
& =&  T. \nonumber
\end{eqnarray}
Remark that the assumption that any square of the form \eqref{Regular-pushout} is a Goursat pushout has been used in the third equality, where it has been applied to the diagram
$$	\xymatrix{
		R  \ar@{->>}[r]^{\overline{f}} \ar@<-2pt>[d]_{r_1} & T \ar@<-2pt>[d]_{t_1}\\
		X   \ar@{->>}[r]_f \ar@<-2pt>[u]_{} & {Y.}  \ar@<-2pt>[u] }
$$
\end{proof}
To conclude this short introduction to Goursat categories we give a characterization of those varieties of universal algebras which are $3$-permutable by using the notion of Goursat pushout. This proof, originally discovered in \cite{HM}, has a categorical version which has first been given in \cite{GR}.

When $\mathbb V$ is a variety of universal algebras, we shall denote by $X = F(1)$ the free algebra on the one-element set. 
\begin{theorem}\label{categ-Goursat}
For a variety $\mathbb V$ of universal algebras the following conditions are equivalent: 
\begin{enumerate}
\item $\mathbb V$ is $3$-permutable: for any pair $R$, $S$ of congruences on any algebra $A$ in $\mathbb V$ one has the equality $$R \circ S \circ R = S \circ R \circ S;$$
\item the theory of $\mathbb V$ contains two quaternary operations $p$ and $q$ satisfying the identities $$p(x,y,y,z) =x, \quad q(x,y,y,z) = z, \quad p(x,x,y,y) = q(x,x,y,y). $$
\end{enumerate}
\end{theorem}
\begin{proof}
$(1) \Rightarrow (2)$ Consider the commutative diagram 
$$	\xymatrix@=45pt{
		X + X+ X+X   \ar@{->>}[r]^{1 + \nabla_2 + 1} \ar@<-2pt>[d]_{\nabla_2 + \nabla_2 } & X +X+X  \ar@<-2pt>[d]_{\nabla_3}\\
		X+X   \ar@{->>}[r]_{\nabla_2} \ar@<-2pt>[u]_{i_2 + i_1} & {X}  \ar@<-2pt>[u]_{i_2} }
$$
where $\nabla_k$ is the codiagonal from the $k$-indexed copower of $X$ to $X$ (for $k \in \{2,3 \}$). The vertical arrows $\nabla_2 + \nabla_2$ and $\nabla_3$ are split epimorphisms, whereas the horizontal arrows are regular epimorphisms, so that the diagram is a Goursat pushout by Proposition \ref{G-pushout}. It follows that the unique morphism $$\overline{1 + \nabla_2 + 1} \colon \mathsf{Eq}(\nabla_2 + \nabla_2) \rightarrow \mathsf{Eq}(\nabla_3)$$ in $\mathbb V$ making the diagram
\begin{equation}
	\xymatrix@=40pt{
		{  \mathsf{Eq}(\nabla_2 + \nabla_2) }  \ar@{->>}[r]^-{\overline{1 + \nabla_2 + 1}} \ar@<-2pt>[d]_{p_1}  \ar@<6pt>[d]^{p_2}&  \mathsf{Eq}(\nabla_3)\ar@<-2pt>[d]_{p_1}\ar@<6pt>[d]^{p_2}\\
		X+ X+ X+X   \ar@{->>}[r]_{1 + \nabla_2 + 1 }  &   X +X+X }
\end{equation}
commute is a regular epimorphism (here $p_1$ and $p_2$ are the kernel pair projections), thus it is surjective. Observe that the terms $p_1 (x,y,z) =x$ and $p_3(x,y,z) = z$ are identified by $\nabla_3$, so that $(p_1, p_3) \in \mathsf{Eq}(\nabla_3)$. The surjectivity of $\overline{1 + \nabla_2 + 1}$ then implies that there are terms $(p,q) \in \mathsf{Eq}(\nabla_2 + \nabla_2)$ such that  $\overline{1 + \nabla_2 + 1} (p,q) = (p_1, p_3)$. This latter property means exactly that 
$$p(x,y,y,z) =x, \quad q(x,y,y,z) = z,$$
while the fact that $(p,q) \in \mathsf{Eq}(\nabla_2 + \nabla_2)$ gives the identity
$$p(x,x,y,y) = q(x,x,y,y). $$
For the converse implication, take $R$ and $S$ two congruences on an algebra $A$ in $\mathbb V$, and let us show that  $R \circ S \circ R \le S \circ R \circ S$. For $(a,b) \in R \circ S \circ R$, let $x$ and $y$ be such that $(a,x) \in R$, $(x,y) \in S$ and $(y,b)\in R$. Then the fact that $(a,a), (x,a), (y,b), (b,b)$ are in $R$ implies that both $(p(a,x,y,b), p(a,a,b,b))$ and $(q(a,x,y,b), q(a,a,b,b))$ are in $R$. Since $p(a,a,b,b) = q(a,a,b,b)$ we deduce that $(p(a,x,y,b), q(a,x,y,b)) \in R$. On the other hand, the elements $(a, a),(x, x),(x, y), (b, b)$ are all in $S$ so that $(p(a, x, x, b),p(a, x, y, b))\in S$, $(q(a, x, x, b),q(a,
x, y, b)) \in S$, hence $(a, p(a, x, y, b))$ and $ (b, q(a, x, y, b))$ are both in $S$. We then observe that 
\begin{eqnarray}
(a, p(a, x, y, b)) &\in S & \nonumber \\
(p(a, x, y, b), q(a, x, y, b)) &\in R&  \nonumber \\
(q(a, x, y, b), b) & \in S& \nonumber
\end{eqnarray}
we conclude that $(a,b)$ belongs to $S \circ R \circ S$. It then follows that $R \circ S \circ R = S \circ R \circ S$, as desired.
\end{proof}
\begin{exe}\label{Implication}
A typical example of $3$-permutable variety, thus of a Goursat category, is provided by the variety $\mathsf{ImplAlg}$ of \emph{implication algebras} \cite{Abb}. The algebraic theory of the variety $\mathsf{ImplAlg}$ has a binary operation such that \begin{enumerate}
\item[(A)] $(xy)x=x$,
\item[(B)] $(xy)y= (yx)x$,
\item[(C)] $x(yz) = y(xz)$.
\end{enumerate}
As explained in \cite{GU}, to see that  $\mathsf{ImplAlg}$ is $3$-permutable, one first checks that the term $xx$ is a constant: indeed, the identities
\begin{eqnarray}
xx &=& [(xy)x]x   \qquad  \qquad \quad  \text{(by (A)) } \nonumber \\
&=& [(x(xy)](xy) \qquad   \qquad \text{(by (B))}  \nonumber \\
& =& x [[x(xy)]y] \qquad  \qquad \text{(by (C))} \nonumber  \\
& =&  x[[((xy)x(xy)]y]   \qquad   \text{(by (A))} \nonumber  \\
&=& x[(xy)y]  \qquad   \qquad\quad \text{(by (A))}  \nonumber  \\
&=& (xy)(xy) \qquad   \qquad \quad \text{(by (C))} \nonumber
\end{eqnarray}
imply that $$xx = [x(yy)][x(yy)] = [y(xy)][y(xy] =yy,$$
and one denotes such an equationally defined constant by $1$. This notation is justified by the fact that
$$1y  = (yy)y =y = y1.$$ 
One then verifies that the terms $p(x,y,z,u)= (zy)x$ and $q(x,y,z,u)=(yz)u$ are such that 
$$p(x,y,y,z) = (yy)x = 1x=x,$$
$$q(x,y,y,z)= (yy)z=1z=z, $$
and $$p(x,x,z,z) = (zx)x = (xz)z = q(x,x,z,z).$$ 
\end{exe}
\begin{rem}
Note that one can give a proof of the Mal'tsev theorem characterizing $2$-permutable varieties by using some categorical arguments similar to the ones in Theorem \ref{categ-Goursat}. This was first observed in \cite{CP} and, more recently, in \cite{BGJ}.
\end{rem}
\begin{rem}
A wide generalization of Theorem \ref{categ-Goursat} was obtained by P.-A Jacqmin and D. Rodelo in \cite{JR}, where a categorical approach to $n$-permutability was deve-loped. Thanks to their approach the authors have been able to characterize the property of $n$-permutability in terms of some specific stability properties of regular epimorphisms, which extend the one considered in \cite{GR3} to study Goursat categories.
\end{rem}
\subsection*{Diagram lemmas and Goursat categories}
We conclude these notes by mentioning a connection between the validity of some suitable \emph{diagram lemmas} and the permutability conditions on a regular category considered above. The classical $3\times 3$-Lemma in abelian categories \cite{Freyd} has been extended to several non-additive context by various authors (see \cite{Bourn2, ZJ}, for instance). An original extension to a non-pointed context was first established by D. Bourn in the context of regular Mal'tsev categories \cite{Bourn}. The main point in order to formulate the $3 \times 3$-Lemma in a category which does not have a $0$-object is to replace the classical notion of short exact sequence with the notion of exact fork: a diagram of the form 
$$
\xymatrix{ R \ar@<-2pt>[r]_{r_2}  \ar@<2pt>[r]^{r_1} & X \ar@{->>}[r]^f & Y
}
$$
is an \emph{exact fork} if and only if $(R, r_1,r_2)$ is the kernel pair of $f$, and $f$ is the coequalizer of $r_1$ and $r_2$.
With this notion at hand the appropriate way of expressing the $3 \times 3$-Lemma is then the following, which is called the \emph{denormalized $3 \times 3$-Lemma}: given any commutative diagram 
\begin{equation}\label{denormalized}
\xymatrix@=40pt{
\mathsf{Eq}(a)  \ar@<-2pt>[r]_{z_1} \ar@<2pt>[r]^{z_2}  \ar@<-2pt>[d]_{a_1} \ar@<2pt>[d]^{a_2}  & \mathsf{Eq}(b) \ar@<-2pt>[d]_{b_1}     \ar@<2pt>[d]^{b_2}  \ar[r]^z & \mathsf{Eq}(c) \ar@<-2pt>[d]_{c_1}     \ar@<2pt>[d]^{c_2}  \\
		\mathsf{Eq}(y) \ar@{->>}[d]_a  \ar@<-2pt>[r]_{y_2}  \ar@<2pt>[r]^{y_1}& {A}   \ar@{->>}[d]^b  \ar@{->>}[r]_y & C  \ar@{->>}[d]^c \\
		K  \ar@<-2pt>[r]_{k_2}  \ar@<2pt>[r]^{k_1}& {B}   \ar[r]_x & D
}
\end{equation}
in $\mathbb C$ such that
\begin{itemize}
\item $y_i \circ a_j = b_j \circ z_i$, $y\circ b_i = c_i \circ z$, $b\circ y_i = k_i \circ a$, $x \circ b = c \circ y$ (for $i,j \in \{1,2\}$),
\item the three columns and the middle row are exact forks,
\end{itemize}
then the upper row is an exact fork if and only if the lower row is an exact fork. \\
S. Lack observed in \cite{Lack} that this \emph{denormalized $3\times3$-Lemma} holds not only in regular Mal'tsev categories (as observed by D. Bourn \cite{Bourn}) but also in Goursat categories. Later on it turned out that the validity of the denormalized $3\times3$-Lemma actually characterizes Goursat categories among regular ones:
\begin{theorem}\cite{Lack, GR}
For a regular category $\mathbb C$ the following conditions are equivalent:
\begin{enumerate}
\item $\mathbb C$ is a Goursat category;
\item if the lower row in a diagram \eqref{denormalized} is an exact fork then the upper row is an exact fork;
\item if the upper row in a diagram \eqref{denormalized} is an exact fork then the lower row is an exact fork;
\item the \emph{denormalized $3\times3$-Lemma} holds in $\mathbb C$: the lower row is an exact fork if and only if the upper row is an exact fork.
\end{enumerate}
\end{theorem}
We would like to point out that both the calculus of relations and the notion of \emph{Goursat pushout} play a central role in the proof of this result.
Note that a unification of both the classical $3\times3$-Lemma and of the denormalized one in the context of star-regular categories is also possible \cite{GJR}. Further results linking the Goursat property to natural conditions appearing in universal algebra - in relationship to congruence modularity - have been investigated in \cite{GRT} (see also the references therein).
 Finally, let us mention that also Mal'tsev categories can be characterized via a suitable diagrammatic condition that is stronger than the denormalized $3\times3$-Lemma, called the \emph{Cuboid Lemma} \cite{GR2}. 


\begin{thebibliography}{10}
\bibitem{Abb} J.C. Abbott, \emph{Algebras of implication and semi-lattices} {S\'eminaire Dubreil. Alg\`ebre et th\'eorie des nombres,} 20 (2) (1966-1967), exp. no. 20, 1-8.
\bibitem{Barr} M. Barr, P. A. Grillet and D. H. van Osdol, \emph{Exact categories and categories of sheaves}, Springer Lecture Notes in Mathematics 236, Springer-Verlag, 1971.
\bibitem{Borceux} F. Borceux, \emph{Handbook of Categorical Algebra. 2. Categories and structures.} Encycl. of Math. and its Applications 51, Cambridge University Press, 1994.
\bibitem{BC} F. Borceux, M.M. Clementino, \emph{Topological semi-abelian algebras} {Adv. Math.} {190}, (2005) 425--453.
\bibitem{Bourn2} D. Bourn, \emph{$3 \times 3$ Lemma and protomodularity}, J. Algebra 236 (2001) 778-795.
\bibitem{Bourn} D. Bourn, \emph{The denormalized $3 \times 3$-Lemma}, Journal of Pure and Applied Algebra, 177, Issue 2 (2003) 113-129.
\bibitem{BG} D. Bourn and M. Gran, \emph{Regular, Protomodular and Abelian Categories}, in Categorical Foundations - Special Topics in Order, Topology, Algebra and Sheaf Theory, Encycl. Math. Appl. 97, Cambridge Univ. Press, 165-211, 2004.
\bibitem{BGJ} D. Bourn, M. Gran and P.-A. Jacqmin, \emph{On the naturalness of Mal'tsev categories,} Preprint (2019) arXiv:1904.06719, to appear in Outstanding Contributions to Logic, Springer.
\bibitem{Buch} D. Buchsbaum, \emph{Exact categories and duality}, Trans. Amer. Math. Soc. 80, 1-34, 1955.
\bibitem{CKP} A. Carboni, G.M. Kelly and M.C. Pedicchio, {\em Some remarks on Mal'tsev and Goursat categories}, Appl. Categ. Structures, 4 (1993) 385--421.
\bibitem{CLP} A. Carboni, J. Lambek and M.C. Pedicchio, {\em Diagram chasing in Mal'cev categories}, Journal of Pure and Applied Algebra, 69 (1990), 271--284.
\bibitem{CP}
	A. Carboni and M.C. Pedicchio, {\em A new proof of the Mal'cev theorem}, Categorical studies in Italy (Perugia, 1977), \textit{Rend. Circ. Mat. Palermo} Vol. 2 Suppl. No. 64 (2000) 13--16.
\bibitem{Freyd} P.J. Freyd, \emph{Abelian categories. An introduction the the theory of functors.} Harper's Series in Modern Mathematics, New York, 1964.
\bibitem{Gran} M. Gran, \emph{Central extensions and internal groupoids in Maltsev categories}, Journal of Pure and Applied Algebra, 155, (2001) 139--166.
\bibitem{Gran-EPFL} M. Gran, \emph{Notes on regular, exact and additive categories}, notes for a mini-course given at the \emph{Summer School on Category Theory and Algebraic Topology}, Ecole Polytechnique F\'ed\'erale de Lausanne, September 2014.
\bibitem{GJR} M. Gran, Z. Janelidze and D. Rodelo, \emph{$3 \times 3$-Lemma for star-exact sequences}, Homology, Homotopy Appl. 14-2 (2012) 1-22.
\bibitem{GR} M. Gran, D. Rodelo, \emph{A New Characterisation of Goursat Categories}, Appl. Categ. Struct. 20 (2012) 229--238. 
\bibitem{GR2} M. Gran and D. Rodelo, {\em The Cuboid Lemma and Mal'tsev categories}, App. Categ. Struct. 22 (5-6) (2014) 805--816.
\bibitem{GR3} M. Gran and D. Rodelo, {\em Beck-Chevalley condition  and Goursat categories}, Journal of Pure and Applied Algebra 221 (2017) 2445--2457, 
\bibitem{GRT} M. Gran, D. Rodelo and I. Tchoffo Nguefeu, {\em Variations of the Shifting Lemma and Goursat categories}, Algebra Univers. 80:2 (2019).
\bibitem{GRos} M. Gran and J. Rosick\'y, \emph{Semi-abelian monadic categories}, Theory Appl. Categ. 13 (6), 2004, 106-113.
\bibitem{GSV} M. Gran, F. Sterck and J. Vercruysse, \emph{A semi-abelian extension of a theorem by Takeuchi}, J. Pure Appl. Algebra 223, (2019) 4171--4190.
\bibitem{Gro} A. Grothendieck, \emph{Technique de construction en g\'eom\'etrie analytique. IV. Formalisme g\'en\'eral des foncteurs repr\'esentables.}, S\'em. Henri Cartan 13 (1), exp. 11,  (1962) 1-28.
\bibitem{GU}  H.P. Gumm, A. Ursini, \emph{Ideals in universal algebra}, Algebra Univers. 19 (1984) 45-54.
\bibitem{HM} J. Hagemann and A. Mitschke, \emph{On n-permutable congruences}, Algebra Univers. 3 (1973) 8-12.
\bibitem{JR} P.-A. Jacqmin and D. Rodelo, \emph{Stability properties characterising $n$-permutable categories}, Theory Appl. Categ.  32 (45) (2017) 1563--1587.
\bibitem{ZJ}  Z. Janelidze, \emph{The pointed subobject functor, $3 \times 3$ lemmas and subtractivity of spans}, Th. Appl. Categ. 23, 221-242, 2010.
\bibitem{Johnstone} P.T. Johnstone, \emph{Stone spaces}, Cambridge Studies in Adv. Mathematics 3, Cambridge University Press 1982.
\bibitem{Johnstone-elephant} P.T. Johnstone, \emph{Sketches of an Elephant: A Topos Theory Compendium}, Oxford Logic Guides, vol. 43, Oxford Univ. Press (2002), Vol. 1.
\bibitem{J-P} P.T. Johnstone and M.C. Pedicchio \emph{Remarks on continuous Mal'cev algebras}, Rend. Ist. Matem. Univ. Trieste 25, 1995, 277-287.
\bibitem{Lack} S. Lack, {\em The 3-by-3 lemma for regular Goursat categories}, Homology, Homotopy, and Applications 6 (1) (2004) 1--3.
\bibitem{Mal}
	A.I. Mal'tsev, \emph{On the general theory of algebraic systems}, {Matematicheskii Sbornik, N.S.} 35 (77) (1954), 3--20.
\bibitem{Meisen} J. Meisen, \emph{Relations in categories}, {McGill University thesis} (1972).
\bibitem{PP} J. Picado and A. Pultr, \emph{Notes on point-free topology}, preprint (2020).
\bibitem{Riguet} J. Riguet, {\em Relations binaires, fermetures, correspondances de Galois}, {Bulletin de la Soci\'et\'e Math\'ematique de France} {76} (1948), 114--155.

\bibitem{Smith} J.D.H. Smith, {\em Mal'cev varieties}, Lecture Notes in Mathematics, 554 (1976).
\end{thebibliography}
\end{document}